\RequirePackage[l2tabu, orthodox]{nag}
\documentclass[11pt]{amsart}

\author{Layne Hall}
\address{School of Mathematics and Statistics, The University of Melbourne, Victoria 3010 Australia}
\email{lrhall@student.unimelb.edu.au}
\author{Andy Hammerlindl}
\address{School of Mathematical Sciences, Monash University, Victoria 3800 Australia}
\urladdr{ http://users.monash.edu.au/~ahammerl/} 
\email{andy.hammerlindl@monash.edu}
\title{Dynamically incoherent surface endomorphisms}
\newcommand\funding[1]{%
	\begingroup
	\renewcommand\thefootnote{}\footnote{#1}%
	\addtocounter{footnote}{-1}%
	\endgroup
}
\providecommand{\keyword}[1]
{\textbf{Keywords:} #1}

\usepackage{amssymb}
\usepackage{amsmath}
\usepackage{amsthm}
\usepackage{mathtools}
\usepackage{float}
\usepackage{bold-extra}
\usepackage[font={small}]{caption}
\usepackage{afterpage}
\usepackage{subcaption}
\usepackage{verbatim}
\usepackage[inline]{asymptote}
\usepackage{pgfplots}
\usepackage{csquotes}
\usepackage{bm}
\usepackage{enumerate}
\usepackage{textcomp}
\usepackage{titlesec}
\usepackage{nameref}
\usepackage{lipsum}
\usepackage{enumitem}
\usepackage{mathrsfs}
\usepackage[capitalise]{cleveref}

\theoremstyle{definition}

\theoremstyle{Theorem}
\newtheorem{thm}{Theorem}

\newtheorem{theorem}{Theorem}
\newtheorem{corollary}[theorem]{Corollary}

\newtheorem{prop}[thm]{Proposition}
\newtheorem{lem}[thm]{Lemma}

\theoremstyle{remark}

\DeclareMathOperator{\eps}{\varepsilon}

\newcommand{\tilf}{\tilde{f}}

\DeclareMathOperator{\bbT}{\mathbb{T}}
\DeclareMathOperator{\bbR}{\mathbb{R}}

\DeclareMathOperator{\bbZ}{\mathbb{Z}}
\DeclareMathOperator{\Cone}{\mathcal{C}}
\DeclareMathOperator{\Bone}{\mathcal{B}}

\DeclareMathOperator{\bbS}{\mathbb{S}}
\def\restrict#1{\raise-.5exhbox{\ensuremath|}_{#1}}

\newcommand{\sint}{\operatorname{int}}

\numberwithin{thm}{section}
\numberwithin{equation}{section}
\numberwithin{figure}{section}

\titleformat{\title}
{\normalfont\scshape\center}{\title}{1em}{}
\titleformat{\section}
{\normalfont\scshape\center}{\thesection}{1em}{}
\titleformat{\subsection}
{\normalfont\bfseries}{\thesubsection}{1 em}{}
\usepackage{fancyhdr}

\usepackage[
backend=biber,
style=alphabetic,
natbib=true,
url=false, 
doi=true,
eprint=false,
maxbibnames=9,
maxcitenames=9
]{biblatex}
\begin{document}
	\maketitle
	\begin{abstract}
	We explicitly construct a dynamically incoherent partially hyperbolic endomorphisms of $\bbT^2$ in the homotopy class of any linear expanding map with integer eigenvalues. These examples exhibit branching of centre curves along countably many circles, and thus exhibit a form of coherence that has not been observed for invertible systems.
	\\
	
	\noindent \keyword{Partial hyperbolicity, Non-invertible dynamics, Dynamical coherence.}
	\end{abstract}
	\funding{This work was partially funded by the Australian Research Council.}
	\section{Introduction}\label{sec:intro}
	
	Understanding the integrability of the centre direction is critical for classifying partially hyperbolic dynamics. Non-invertible surface maps demonstrate a broader array of dynamics than their invertible counterparts, for instance, while partially hyperbolic surface diffeomorphisms are dynamically coherent, examples in \cite{enco} and \cite{heshiwang} show that there exist incoherent non-invertible maps. In this paper, we introduce a periodic centre annulus as a mechanism for incoherence of partially hyperbolic surface endomorphisms, and use this to construct incoherent surface endomorpshims which are homotopic to linear expanding maps. The centre curves of these examples behave unlike those of the currently known maps on $\bbT^2$ and diffeomorphisms in higher dimensions.
	
	We begin the statement of our results by recalling the definition of partial hyperbolicity. For non-invertible maps, this is most naturally given in terms of cone families. A \emph{cone family} $\Cone \subset TM$ consists of a closed convex cone $\Cone(p) \subset T_p M$ at each point $p \in M$. A cone family is \emph{$Df$-invariant} if $D_p f \left(\Cone(p)\right)$ is contained in the interior of $\Cone(f(p))$ for all $p\in M$. A map $f:M\to M$ is a \emph{(weakly) partially hyperbolic endomorphism} if it is a local diffeomorphism and it admits a cone family $\Cone^u$ which is $Df$-invariant and there is $k>0$ such that $1 < \| Df^k v^u \|$ for all unit vectors $v^u \in \Cone^u$. In general, an unstable cone family in the non-invertible setting does not imply the existence of an invariant splitting. However, it does imply the existence of a centre direction, that is, a continuous $Df$-invariant line field $E^c\subset TM$ \cite[Section 2]{cropot2015lecture}. For $M$ an orientable closed surface, the existence of $E^c$ implies that $M=\bbT^2$.
	
	The homotopy class of a partially hyperbolic surface endomorphism $f$ plays a fundamental role in the existing classification results. The endomorphism $f$ induces a homomorphism of the fundamental group $\pi_1(\bbT^2)\simeq \bbZ^2$ which can be given by an invertible integer-entried $2\times2$ matrix $A$. The matrix $A$ induces a linear automorphism on $\bbT^2$ which we call the \emph{linearisation} of $f$. Pertinently, $f$ is homotopic to its linearisation. It is useful to categorise the linearisation into three types based on the eigenvalues $\lambda_1$ and $\lambda_2$ of the matrix inducing $A$ which we refer to as follows:
	\begin{itemize}
		\item if $|\lambda_1| < 1 < |\lambda_2|$, we say $A$ is \emph{hyperbolic} if, 
		\item if $1<|\lambda_1| \leq |\lambda_2|$, we say $A$ is \emph{expanding}, and
		\item if $1=|\lambda_1|\leq|\lambda_2|$, we say $A$ is \emph{non-hyperbolic}.
	\end{itemize}
	We say a partially hyperbolic endomorphism of $\bbT^2$ is \emph{dynamically coherent} if there exists an $f$-invariant foliation tangent to $E^c$. Otherwise, we say it is \emph{dynamically incoherent}. A closely related property which is sufficient for coherence is unique integrability: the centre direction $E^c$ is said to be uniquely integrable if there is a unique $C^1$ curve tangent to $E^c$ through every point.
	
	Every endomorphism which has hyperbolic linearisation is both dynamically coherent and leaf conjugate to its linearisation \cite{enco}. Both \cite{heshiwang} and \cite{enco} show this does not hold in general by constructing incoherent endomorphisms, both of which have non-hyperbolic linearisation. We are naturally left with the question: how does the centre direction behave in the case of an expanding linearisation? Our first result addresses this.
	
	\begin{theorem}\label{thm:example}
		There exists a partially hyperbolic endomorphism $f:\bbT^2\to \bbT^2$ which is homotopic to an expanding map and whose centre direction is not uniquely integrable. Moreover, the centre direction of $f$ is uniquely integrable on an open and dense subset of $\bbT^2$, but is not uniquely integrable at each point in a countably infinite family of circles.
	\end{theorem}
	
	 We will prove \cref{thm:example} by constructing a geometric mechanism called an \emph{invariant centre annulus}, which is an immersed open annulus $X\subset\bbT^2$ such that $f(X)=X$ and whose boundary, which necessarily consists of either one or two disjoint circles, is tangent to the centre direction. Note that the case when the boundary is one circle is precisely when the closure of the annulus is $\bbT^2$. If $X$ is an invariant annulus for some positive iterate $f^n$ of $f$, then we call $X$ a \emph{periodic centre annulus}. The annulus in the example is constructed by taking a linear expanding map, opening up an invariant circle to obtain an invariant annulus, and then applying a shear on the interior of the invariant annulus. Partial hyperbolicity of the example is established by the construction of an unstable cone-family, and the invariant annulus on which the shear was applied becomes an invariant centre annulus. The centre direction is uniquely integrable on this open annulus, but not along its boundary, and so we observe the behaviour of $E^c$ by taking preimages of the annulus. The complete construction of this example is carried out in \cref{sec:eg}.

  We present a very specific example in \cref{thm:example} for concreteness, but in \cref{sec:gen} we generalise this procedure to construct a family of examples to prove the following result.
	\begin{theorem}\label{thm:allclasses}
		Let $A:\bbT^2\to\bbT^2$ be a linear map with integer eigenvalues and at least one eigenvalue greater than $1$. Then there exists a partially hyperbolic surface endomorphism which is homotopic to $A$ and is dynamically incoherent.
	\end{theorem}
	For the non-hyperbolic case, where $|\lambda_2|>|\lambda_1|=1$, examples establishing \cref{thm:allclasses} have already been constructed in \cite{heshiwang} and \cite{enco}. Thus, we prove the result by constructing incoherent examples homotopic to any expanding linear map with integer eigenvalues.
	
    \medskip{}

	To contrast the difference between the invertible and non-invertible settings,
    we briefly survey the known examples of dynamical incoherence
    in the case of diffeomorphisms.
    For partially hyperbolic diffeomorphisms with a centre bundle of dimension 2
    or higher, it is possible to construct examples with smooth centre bundles
    that are not integrable. In this smooth setting, the integrability, or lack
    thereof, is given by the involutivity condition of Frobenius.
    See \cite{bw} and \cite{hammerlindl-jmd} for further details.
    
    For the case of one-dimensional centre bundle, the question was open much
    longer.
    Since $C^1$ vector fields are always integrable,
    a dynamically incoherent example here would, by necessity,
    have a centre direction which is not $C^1$, or even Lipschitz.
    Rodriguez-Hertz, Rodriguez-Hertz, and Ures
    constructed an example on the 3-torus,
    using an invariant 2-torus tangent to the centre-unstable direction
    \cite{RHRHU}.
    In this example, the non-integrability of the centre direction occurs
    only at this 2-torus and the centre direction is smooth and therefore 
    uniquely integrable everywhere else on the 3-torus.
    In fact, any partially hyperbolic system on the 3-torus,
    can have only finitely many embedded 2-tori tangent to $E^{cs}$ or $E^{cu}$
    and the centre direction is integrable outside these regions \cite{clab}.
    In any dimension,
    a partially hyperbolic diffeomorphism can have 
    only finitely many compact submanifolds tangent either to
    $E^{cs}$ or $E^{cu}$ \cite{prox}.
    
    More recently, new examples of partially hyperbolic diffeomorphisms
    have been discovered on the unit tangent bundles of
    surfaces of negative curvature \cite{BGHP}.
    For certain homotopy classes,
    these systems are dynamically incoherent.
    Further, these examples have unique branching foliations tangent to the $E^{cs}$ and $E^{cu}$
    directions and the branching (i.e.,~merging of distinct leaves)
    occurs at a dense set of points. 
    The dynamical incoherence of such examples may therefore be regarded as a global phenomenon.
    
    In the examples we construct to prove \cref{thm:allclasses}, the branching of the centre direction occurs at an infinite collection of
    circles tangent to the centre and the closure of this collection gives a
    lamination consisting of uncountably many circles.
    Moreover, outside this lamination, the centre direction is uniquely integrable
    and consists of lines.
    The branching is therefore not global, nor is it confined to a submanifold.
    This type of dynamical incoherence is possible due to the non-invertible nature
    of partially hyperbolic endomorphism.
	
    The behaviour of the examples in \cref{thm:allclasses} is also distinct from the previously known incoherent endomorphisms on $\bbT^2$. Namely, the examples in \cite{heshiwang} and \cite{enco} both have centre curves branch on the boundary of a submanifold, and so are analogous to diffeomorphisms in dimension 3. Furthermore, these preexisting examples admit invariant unstable directions much like the invertible setting, while the examples of the current paper do not.
	
	\cref{thm:allclasses} also has an immediate consequence to a previously unanswered question about which homotopy classes admit endomorphisms. Given a partially hyperbolic endomorphism of $\bbT^2$, its linearisation $A$ is given by an invertible integer-entried $2\times2$ matrix, and existing techniques show that $A$ must have real eigenvalues. If the eigenvalues have distinct magnitude, then $A$ itself induces a partially hyperbolic endomorphism. This is not true if the eigenvalues have equal magnitude, and it was unknown if there can exist a partially hyperbolic endomorphism which is homotopic to such an $A$. In particular, the question of existence had been posed for when $A$ is twice the identity, and this is now answered by an immediate corollary of \cref{thm:allclasses}.
	\begin{corollary}\label{corollary}
		There exists a partially hyperbolic surface endomorphism which is homotopic to 
		$ \begin{psmallmatrix}
			2 & 0 \\
			0 & 2
			\end{psmallmatrix}.
	    $
	\end{corollary}
	
	Finally, there are now two qualitatively different forms of incoherent partially hyperbolic surface endomorphisms: those constructed in this paper, and those in \cite{heshiwang} and \cite{enco}. In preparation is a classification of partially hyperbolic surface endomorphisms in \cite{end}, where it is shown that any incoherent example is akin to one of these examples.
	
	\section{Concrete example}\label{sec:eg}
	In this section we construct an explicit example to prove \cref{thm:example}.
    The resulting endomorphism will be homotopic to
		$ \begin{psmallmatrix}
			4 & 0 \\
			2 & 3
			\end{psmallmatrix}.
	    $
    In the next section, we show how to generalize the construction
    to other homotopy classes.

    \pagebreak[3] 

    For small $0<a<1$, let $g:\bbS^1\to\bbS^1$ be a smooth, monotone map which:
	\begin{itemize}
		\item is homotopic to the quadrupling map $x\mapsto 4x$,
		\item has fixed points $x=-a,\,0,\, a$ and no other fixed points in $[-a,a]$,
		\item satisfies $g'(0)<1$,
		\item is linear on the complement of $(-a,a)$, and
		\item $g'(x)\geq 4$ for $x\in\bbS\setminus[-a,a]$.
	\end{itemize} 
	Define $f_0:\bbT^2\to\bbT^2$ by $f_0(x,y)=(g(x),3y)$. Then $f_0$ is a map which is homotopic to the linear map $(x,y)\mapsto(4x,3y)$, and fixes the annulus $[-a,a]\times\bbS^1$. Let $\varphi:[-1/2,1/2]\to[-1,1]$ be a smooth, monotone function which is odd, that is $\varphi(-x)=-\varphi(x)$, such that $\varphi'$ is zero off of $(-a,a)$ while satisfying $\varphi(-a)=-1$, $\varphi(0)=0$, and $\varphi(a)=1$. Then $\varphi$ defines a map to $\bbS^1\to \bbS^1$ which is a shear that is seen only inside the interval $(-a,a)$. We further take $\varphi'(0)>1$, noting that we can take $\varphi'(0)$ to be arbitrarily large by taking the support of $\varphi$ to be smaller. It will also be convenient to take $\varphi$ and $g$ to be linear on a small neighbourhood about each of the points $-a$, $0$ and $a$. Graphs of suitable functions $g$ and $\varphi$ are shown in \cref{fig:graphs}.
	\begin{figure}
	\centering
     \begin{subfigure}[b]{0.3\textwidth}
     \centering
            \includegraphics{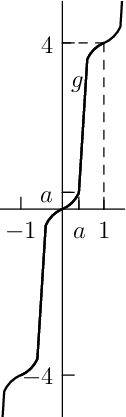}       
     \end{subfigure}
     \begin{subfigure}[b]{0.5\textwidth}
     \centering
                \includegraphics{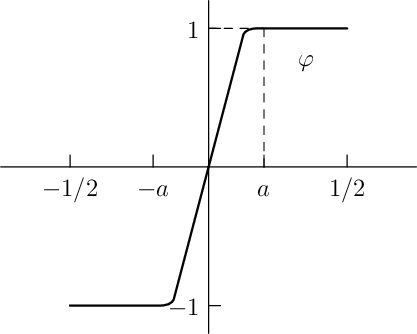}
     \end{subfigure}

        \caption{Graphs of the functions $g$ and $\varphi$ used to define the example $f$.}
        \label{fig:graphs}
\end{figure}
	Our example $f:\bbT^2\to\bbT^2$ is an explicit deformation of $f_0$ defined by $f(x,y) = (g(x),3y+\varphi(x))$. From our construction, we see the linearisation $A:\bbT^2\to\bbT^2$ of $f$ is given by
	\[
	B =
	\left(
	\begin{matrix}
	4 & 0 \\
	2 & 3
	\end{matrix}
	\right),
	\]
	and so $f$ is indeed homotopic to a linear expanding map.
	
	Next we establish that $f$ is indeed partially hyperbolic by building an unstable cone family. The derivative of $f$ at a point $(x,y)\in\bbT^2$ is given by
	\[ D_{(x,y)}f = \left(
	\begin{matrix}
	g'(x) & 0 \\
	\varphi'(x) & 3
	\end{matrix}
	\right).
	\]
	Note that $\Lambda = \{0\}\times\bbS^1\subset\bbT^2$ is an $f$-invariant circle, and that on this circle, the derivative is given by
	\begin{equation*} D_{\Lambda}f = \left(
	\begin{matrix}
	g'(0) & 0 \\
	\varphi'(0) & 3
	\end{matrix}
	\right).
	\end{equation*}
	 Since $g'(0)<1$, then $\Lambda$ is an invariant hyperbolic attractor, and we will use this to define a cone-family on a neighbourhood of $\Lambda$. Let $U_{\Lambda}$ be a small open neighbourhood of $\Lambda$ on which $\varphi'$ and $g'$ are constant. For $p\in U_{\Lambda}$ and $\eps>0$, define a constant cone family $\Cone_{\eps}(p)$ as the cone containing the first quadrant with boundary given by the slopes of $(1,-\eps)$ and $(-\eps,1)$ as depicted in \cref{fig:cones}.
	 
	 \begin{figure}[h]
                        \includegraphics{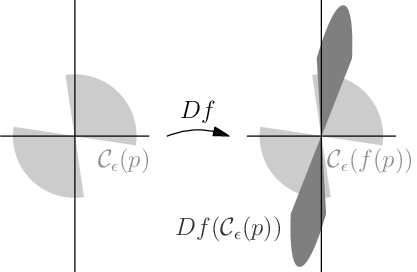}
                        	    \caption{The cone family $\Cone_{\eps}$ with the behaviour as proved in \cref{coneeps}}.
                        	    \label{fig:cones}
    \end{figure}
	\begin{lem}\label{coneeps}
		There is $\eps>0$ such that $\Cone_{\eps}$ is expanded by $Df$, and $Df(\Cone_{\eps}(p))\subset\sint(\Cone_{\eps}(f(p)))$ for all $p\in U_{\Lambda}$.
	\end{lem}
	\begin{proof}
		If $p\in U_{\Lambda}$, then $Df$ is given by the constant matrix $
		\begin{psmallmatrix}
		g'(0) & 0 \\
		\varphi'(0) & 3
		\end{psmallmatrix}$
		 and one can use this to show that $Df(\Cone_{\eps}(p))\subset\sint(\Cone_{\eps}(f(p)))$. We leave the proof of this to the reader.
		
		To show that there is $\eps$ such that $Df$ expands $\Cone_{\eps}$, by continuity, it suffices to show that $Df$ expands all vectors in the first quadrant. So let $(u,v)\in T_p\bbT^2$ be non-zero lying in the first quadrant; that is, $u,v\geq 0$.
        The derivative is
		$
		D_pf (u,v) = (g'(0)u,\varphi'(0)u + 3v)
		$
		with $\varphi'(0)>1$. Then
        \[
            (\varphi'(0)u + 3v)^2
            \ \ge \ 
            (\varphi'(0) u)^2 + (3v)^2
            \ > \ 
            u^2 + v^2,
        \]
        implies that $\|D_pf (u,v)\|>\|(u,v)\|$.
	\end{proof}
	For the remainder of this section, we fix $\eps$, and thus $\Cone_{\eps}$, to be as in the preceding lemma.
	
	Next, consider the compact set $K = \bbT^2\setminus \left((-a,a)\times\bbS^1\right)$. This set is `backward invariant' in the sense that $f^{-1}(K)\subset K$. Then $f$ is linear and expanding on $K$, which will allow us to define a natural unstable cone family on a neighbourhood of $K$. Let $U_K$ be a small open neighbourhood of $K$ which is disjoint from $U_{\Lambda}$ and is such that $\varphi'(x) = 0$ and $g'(x)>3$ for all $(x,y)\in U_K$.
	\begin{lem}\label{bone}
		Suppose that $\Bone_{\delta}$ is a cone family on $U_K$ which for $0<\delta<1$ is given at each point by the cone that contains the horizontal and is bounded by the slopes $(1,\delta)$ and $(1,-\delta)$. Then $\Bone_{\delta}$ is expanded by $Df$, and $Df(\Bone_{\delta}(p))\subset\sint(\Bone_{\delta}(f(p)))$ for all $p\in U_K$ which satisfy $f(p)\in U_K$.
	\end{lem}
	\begin{proof}
		By definition, if $(x,y)\in U_K$, then
		\[
		D_{(x,y)}f = \left(
		\begin{matrix}
		g'(x) & 0 \\
		0 & 3
		\end{matrix}
		\right).
		\]
		Since $g'(x)>3$, $Df_{(x,y)}$ is expanding. Moreover, $Df_{(x,y)}$ preserves both the horizontal and vertical, and expands the horizontal more than the vertical. This implies the result.
	\end{proof}
	 A depiction of a cone family $\Bone_{\delta}$ as in the preceding lemma is shown in \cref{fig:bone}.
	\begin{figure}[h]
	        \includegraphics{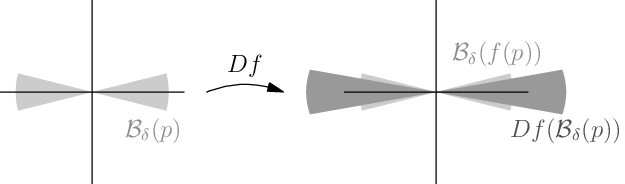}
	    \caption{A cone family $\Bone_{\delta}$ with the properties stated in \cref{bone}.}
	    \label{fig:bone}
	\end{figure}

	 In \cref{fig:zones} we depict the regions $U_{\Lambda}$ and $U_K$ over which we have defined $\Cone_{\eps}$ and $\Bone_{\delta}$. The next step is to `stitch together' these two cone families to obtain a global one. Let $V\subset (-a,a)\times \bbS^1$ be an open strip such that $V \cup U_K =\bbT^2$ and $f(V)\cap U_K = \varnothing$. Note that the definition of the circle map $g$ implies that $f(V)\Subset V$, and so there is $N$ such that $f^N(V)\subset U_{\Lambda}$. We pull back $\Cone_{\eps}$ to define a cone family $\Cone^N$ on $V$ by $\Cone^N(p) = Df^{-N}(\Cone_{\eps}(f^N(p)))$ at $p\in V$.
	 
	 \begin{figure}[h]
        \includegraphics{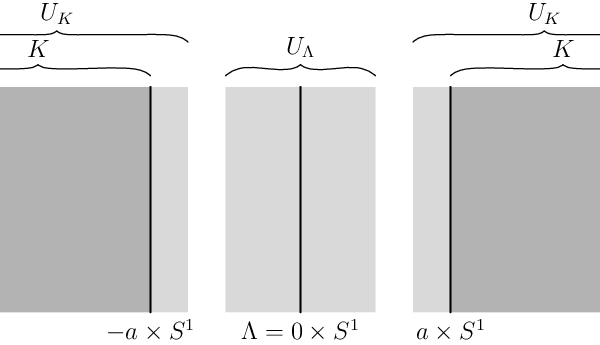}
     \caption{The regions $U_{\Lambda}$ and $U_K$ in $\bbT^2$ we over which we have defined $\Cone_{\eps}$ and $B_{\delta}$. We glue these cone families together across the gap between these regions to obtain $\Cone^u$.}
     \label{fig:zones}
    \end{figure}
	 \begin{lem}\label{lem:conen}
	 	For $p\in V$, the cone $\Cone^N(p)$ is a closed neighbourhood of the horizontal direction.
	 \end{lem}
 	\begin{proof}
 	    We show the result by considering how $Df$ pulls back tangent vectors that lie in the second and fourth quadrants of the tangent space. For $q\in \bbT^2$, choose a point $q'= (x,y) \in f^{-1}(q)$. Given a tangent vector $(u,v)\in T_q \bbT^2$, the tangent vector $(\hat{u},\hat{v}) = D_{q'}f^{-1}(u,v) \in T_{q'}\bbT^2$ is given explicitly by 
 		\[(\hat{u},\hat{v})= D_{q'}f^{-1}(u,v) = \left(\frac{u}{g'(x)}, -\frac{\varphi'(x)}{3g'(x)}u + \frac{1}{3}v\right).\]
  Since $\varphi',g'\geq 0$, then $v<0<u$ implies $\hat{v}<0<\hat{u}$. Similarly, $u<0<v$ implies $\hat{u}<0<\hat{v}$. So $D_{q'}f^{-1}$ maps the second and fourth quadrants into themselves, meaning if we take any cone of tangent vectors in $ T_q\bbT^2$ which encloses the first and third quadrants, the pullback of this cone under $Df$ will also enclose these quadrants in $T_{q'}\bbT^2$.

 		Now given $p\in V$, let $q\in U_{\Lambda}$ be such that $f^N(p)=q$. By definition, $\Cone_{\eps}(q)$ is a neighbourhood of the first and third quadrants, and $\Cone^N(p)$ is $N$ successive pullbacks of $\Cone_{\eps}(q)$ by $Df$. Thus, by induction using the property shown in the preceding paragraph, $\Cone^N(p)$ is a closed neighbourhood of the first and third quadrants. In particular, the claim of the lemma holds.
 	\end{proof}
 	
	\begin{lem}\label{compatible}
		There is a cone family $\Bone$ on $U_K$ such that following hold:
		\begin{itemize}
			\item if $p\in U_K$ and $f(p) \in U_K$, then $ Df \Bone(p)\subset \sint \Bone(f(p))$;
			\item if $p\in U_K$ and $f(p)\in V$, then $Df\Bone(p)\subset\sint \Cone^N(f(p))$;
			\item if $p\in V \cap U_K$, then $\Bone(p)\subset \sint\Cone^N(p)$.
		\end{itemize}
	\end{lem}
	\begin{proof}
	The first property will hold by letting $\Bone=\Bone_{\delta}$ as in \cref{bone} for any small $\delta>0$. Since we can take $\delta$ arbitrarily small, then to prove the second and third claims, it suffices to show that $\Cone^N$ is a closed neighbourhood of the horizontal at each point. This was established in \cref{lem:conen}.
	\end{proof}
	Recall (see \cite{cropot2015lecture}, \S 2.2) that a continuous cone family $\Cone$ is equivalent to a continuous quadratic form $Q$ on the tangent space. In this correspondence, the cone $\Cone(p)$ at a point $p$ is determined by the quadratic $Q_p$ by $\Cone(p) = \{v\in T_p\bbT^2:Q_p(v)\geq 0\}$.
	\begin{lem}\label{glue}
		The map $f$ admits a cone family $\Cone^u$ such that $Df(\Cone^u)\subset\sint \Cone^u$ and which coincides with $\Bone$ on $U_K$ and $\Cone^N$ on $f(V)$.
	\end{lem}
	\begin{proof}
		Let $\Cone^N$ be prescribed by a quadratic form $P$ on $V$, and $\Bone$ by a form $Q$ on $U_K$. Let $\alpha:\bbT^2\to[0,1]$ be a continuous function such that $\alpha(f(V)) = \{0\}$ and $\alpha(U_K) = \{1\}$. Define a continuous cone family $\Cone^u$ by the quadratic form
		\[
		(1-\alpha)P + \alpha Q.
		\]
		Then $\Cone^u$ coincides with $\Cone^N$ on $f(V)$ and $\Bone$ on $U_K$, and it is invariant by \cref{compatible}.
	\end{proof}
	\begin{prop}\label{ph}
		The endomorphism $f$ is partially hyperbolic.
	\end{prop}
	\begin{proof}
		By \cref{glue}, $\Cone^u$ is an invariant cone family. It remains to show that $\Cone^u$ is expanded by $Df^k$ for some $k$. Observe that there is $m>0$ such that if $p\in \bbT^2$, then the orbit of $p$ has at most $m$ points which do not lie in either $U_K$ or $U_{\Lambda}$. We know that $\Cone^u=\Bone_{\delta}$ on $U_K$, while for $q\in U_{\Lambda}$, we have $Df^{N}(\Cone^u(q))= \Cone_{\eps}$. Thus, with the exception of at most $m+N$ points in the orbit of $p$, $Df^n\Cone^u(p)$ lies in either $\Cone_{\eps}$ or $ \Bone_{\delta}$. By \cref{coneeps} and \cref{bone}, $Df$ expands $\Cone_{\eps}$ and $\Bone_{\delta}$. Thus by choosing $k$ sufficiently large, if $v^u\in\Cone^u(p)$ is a unit vector, we have $\|Df^k v^u\| > 1$.
	\end{proof}
    We now know that $f$ is partially hyperbolic, and so it admits an invariant centre direction $E^c$ (see the discussion in the introduction of \cite{enco}). We will later argue that is does not however admit an invariant unstable direction on $\bbT^2$.

    Define an annulus $X = (-a,a)\times\bbS^1$.
	Lift $f$ to a diffeomorphism $\tilde{f}:\bbR^2\to\bbR^2$ which fixes the lift $\tilde{X} = (-a,a)\times\bbR$ of $X$. As the lift of a partially hyperbolic surface endomorphism, the map $\tilde{f}$ admits an invariant splitting $E^c \oplus E^u$, with $E^c$ descending to the centre direction of $f$ on $\bbT^2$ \cite{manepughendo}. Moreover, $\tilde{f}$ is a finite distance from its linearisation $A:\bbR^2\to\bbR^2$, which we remind the reader was defined in the introduction.
	\begin{prop}\label{invtannulus}
		The map $f$ admits an invariant centre annulus $X$, and the centre direction $E^c$ is uniquely integrable on $X$.
	\end{prop}
	\begin{proof}
	The boundary components of the invariant annulus $(-a,a)\times\bbS^1$ are the circles $\{a\}\times\bbS^1$ and $\{-a\}\times\bbS^1$. Restricted to these circles, $Df$ is given by
	\[
	\left(
	\begin{matrix}
	4 & 0 \\
	0 & 3
	\end{matrix}
	\right).
	\]
	
	Thus $\{a\}\times\bbS^1$ and $\{-a\}\times\bbS^1$ are tangent to $E^c$, and so $X = (-a,a)\times\bbS^1$ is an invariant centre annulus.
	
        To see that $E^c$ is uniquely integrable on $X$, note that the restriction of $\tilde{f}$ to $\tilde{X}$ is a diffeomorphism. Under this diffeomorphism, the set $\{0\}\times\bbR$ is an invariant hyperbolic manifold with splitting $E^u\oplus E^s$. Here, $E^s$ corresponds locally to $E^c$. Using classical stable manifold theory, one may then show that $E^c$ is uniquely integrable on a neighbourhood $U$ of $\Lambda$ \cite{HPS}. But as every point in $p \in X$ is in the preimage of some point in $q\in U$, so $E^c$ is uniquely integrable on all of $X$.
	\end{proof}
	We now establish that the centre curves must branch on the boundary of $X$.
	\begin{lem}\label{boxes}
        Let $U\subset \tilde{X}$ be a box of the form $(-a,a)\times (-r_0,r_0)$ for some $r_0>0$. Then there exists $r>0$ such that $\tilde{f}^{-n}(U)\subset (-a,a)\times(-r,r)$ for all $n>0$.
	\end{lem}
	\begin{proof}
		If $U = (-a,a)\times (-r_0,r_0)$, we directly compute $A^{-1}(U)  =  (-a/4,a/4) \times (-r_0/3,r_0/3)$. Since $\tilde{f}$ is a finite distance from $A$ and the strip $\tilde{X} = (-a,a)\times \bbR$ is $\tilde{f}$-invariant, then $\tilde{f}^{-1}(U)\subset (-a,a)\times  (-r_0/3-C,r_0/3+C)$ for some $C>0$. Now we fix $r>r_0$ to satisfy $r>\frac{r}{3}+C$. The preceding calculation implies that $\tilde{f}^{-1}((-a,a)\times(-r,r)) \subset (-a,a)\times(-r,r)$. Since $r>r_0$ then $U\subset (-a,a)\times(-r,r)$ and the result follows by induction.
	\end{proof}
	
	Now we establish that the centre direction is not-uniquely integrable, proving \cref{thm:example}.
	\begin{proof}[Proof of \cref{thm:example}]
		Suppose that $E^c$ is uniquely integrable on all of $\bbR^2$, so that it integrates to a foliation which descends to $\bbT^2$. Consider a small centre curve $J^c \subset \bbR^2$ with one endpoint at the origin and the other endpoint at $p\in(0,a)\times\bbR$, so that $J^c\subset (0,a)\times\bbR$. Under $\tilf^{-n}$, the endpoint at the origin remains fixed. Note that the $x$-coordinate of $p$ lies in the interval $(0,a)$, and so by definition of $f$, the behaviour of $\tilf^{-n}(p)$ is determined by the dynamics of the circle map $g$ on $(0,a)\subset\bbS^1$. Namely, since $a\in\bbS^1$ is a repelling fixed point of $g$, the $x$-coordinate of $\tilf^{-n}(p)$ monotonically approaches $a$, and $\tilf^{-n}(p)$ approaches the vertical line $\{a\}\times\bbR$. By unique integrability of $X$ on $E^c$, we have $\tilf^{-n}(J^c)\subset \tilf^{-(n+1)}(J^c)$, with both of these curves being leaf segments of the leaf of the centre foliation through $(0,0)$. The curve $\{a\}\times\bbR$ is tangent to the centre, so by our assumption of unique integrability, it is necessarily a leaf of the centre foliation, implying that the endpoint of $\tilf^{-n}(J^c)$ cannot converge to a point on $\{a\}\times\bbR$.  Thus, $\tilf^{-n}(J^c)$ must grow unbounded in the vertical direction as $n\to\infty$. But by \cref{boxes}, $\tilf^{-n}(J^c)$ must be uniformly bounded in the vertical direction, giving a contradiction.
		
		Now we prove that the centre curves branch only on a set of countably many annuli. For this, recall that $g$ is homotopic to the circle map $x\mapsto 4x$. Then, a single preimage of the interval $[-a,a]\subset\bbS^1$ under the circle map $g$ consists of $[-a,a]$ and three other disjoint intervals, and we claim that the union of all backward iterates of $[-a,a]$ under $g$ is dense in $\bbS^1$. If this were not the case, there is some interval $I\subset\bbS^1$ whose set of orbits is disjoint from $[-a,a]$. Since $g$ is linear with $g'>1$ on $\bbS^1\setminus [-a,a]$, then $g^n(I)$ cannot be disjoint from $[-a,a]$ for large $n$. Thus, the preimages of $[-a,a]$ must be dense in $\bbS^1$. Then, the preimage of the invariant annulus $X$ under $f$ consists of $X$ and three other disjoint annuli, and $\bigcup_{n\geq 0}f^{-n}(X)$ is dense in $\bbT^2$.
		
		Since $E^c$ is uniquely integrable on $X$, then it is uniquely integrable on $\bigcup_{n\geq 0}f^{-n}(X)$. The centre direction is not uniquely integrable on each boundary circle of $X$, and is thus not uniquely integrable on each of their preimages. Thus $E^c$ behaves as desired.
	\end{proof}
	\begin{figure}[h]
	\includegraphics{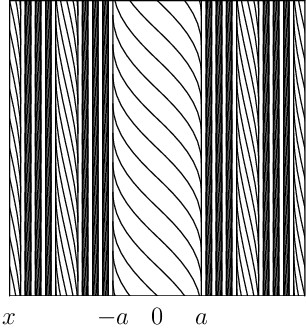}
        \caption{The centre curves of the endomorphism $f$ on $\bbT^2$.}
        \label{fig:tan}
	\end{figure}
	We remark that the endomorphism $f$ does not admit a global invariant unstable direction $E^u$. To see this, consider the fixed point $p=(0,0)\subset \Lambda$. Using our computed expression $D_{\Lambda}f = \begin{psmallmatrix}
		g'(0) & 0 \\
	\varphi'(0) & 3
	\end{psmallmatrix}$, we see that $Df^n(\Cone^u(p))$ becomes an arbitrarily small neighbourhood of the vertical as $n\to\infty$. So, if we suppose that $f$ has an invariant unstable direction $E^u$, then $E^u(p)$ must be vertical. Now consider $q\in K$ be a point such that $f(q) = p$. The unstable cone $\Cone^u$ at $q$ is a small unstable neighbourhood of the horizontal, and this cone is mapped by $Df$ to a small neighbourhood of the horizontal in $T_p\bbT^2$, as seen in \cref{bone}. But $\Cone^u$ is an unstable cone family for $f$, which means that $Df\Cone^u(q)$ must contain $E^u$, contradicting that $E^u(p)$ must be vertical.
	
	 With \cref{thm:example} proved, we now work toward obtaining a clear depiction of the centre curves as shown in \cref{fig:tan}. If $p$ is a point whose orbit is disjoint from $X$, then the orbit of $p$ lies in $K$. For $(x,y)\in K$, recall that
	\begin{equation}\label{eqn:diagonalderiv}
		Df_{(x,y)} = \left(
	\begin{matrix}
	g'(x) & 0 \\
	0 & 3
	\end{matrix}
	\right)
	\end{equation}
	and that $g'(x)$ is a constant greater than $4$, we see that the centre direction on the orbit of such $p$ is vertical. Elsewhere, we can use the following.
	\begin{lem}\label{lem:negativeslope}
		The centre direction $E^c$ on $X$ has negative slope.
	\end{lem}
	\begin{proof}
		Recall that $E^c$ lies in the complement of $\Cone^u$. For $q\in U_{\Lambda}$, the unstable cone is $\Cone_{\eps}(q)$, so all vectors not in $\Cone^u$, including those in the centre direction, have negative slope. Hence the proposition holds at $q$. If $p \in X$, is such that $f(p)=q$ for some $q\in U_{\Lambda}$, then $E^c(p) = Df^{-1}(E^c(q))$. But $E^c(q)$ has negative slope, and it was shown in the proof of \cref{glue} that $D_pf^{-1}$ takes vectors of negative slope to vectors of negative slope for $p\in X$, so the claim also holds at $p$. Since $X\subseteq \bigcup_{n\geq 0}f^{-n}(U_{\Lambda})$, then by continuing in this fashion, the claim holds at all points in $X$.
	\end{proof}
	Since $Df$ is constant given by the equation \cref{eqn:diagonalderiv} on the complement of $X$, then a connected component of $f^{-1}(X)$ which is not itself $X$ is a rescaled copy of $X$. This copy is contracted more strongly horizontally than vertically. This procedure continues while iterating $X$ backwards, and so one can show that the centre curves are as in \cref{fig:tan}.
	
	We conclude by noting that the authors believe that this example $f$ is dynamically incoherent. However, showing this would require more effort than required for our purposes, since it will be easier to prove incoherence of the examples we construct in the next section.
\section{General construction}\label{sec:gen}
In this section, we prove \cref{thm:allclasses} by generalising the construction used to establish \cref{thm:example}. Dynamically incoherent examples for when linearisation has an eigenvalue of magnitude $1$ are constructed in the earlier works \cite{heshiwang} and \cite{enco}, so we only need to consider the case of when the linearisation is expanding. Our approach is eased by first observing we only need to construct an example homotopic to linear maps of a certain form:
\begin{lem}\label{lem:triangle}
	Let $A:\bbT^2\to\bbT^2$ be a linear expanding map with integer eigenvalues $\lambda$ and $\mu$. Then $A$ is conjugate as a map on $\bbT^2$ to a linear map $B: \bbT^2\to\bbT^2$ given by
	\[
	B = 
	\left(
	\begin{matrix}
	\mu & 0 \\
	t & \lambda
	\end{matrix}
	\right)
	\]
	for some $t\in\bbZ$.
\end{lem}
\begin{proof}
	As the eigenvalue $\lambda$ of $A$ is an integer, it has an associated eigenvector $v = (a,b)\in\bbZ^2$ for $a$ and $b$ with $\gcd(a,b)=1$. The lemma will be satisfied if we can find $P \in SL(2,\bbZ)$ such that $Pv = (0,1)$, as we may take $B = PAP^{-1}$. This amounts to solving two coupled equations over $\bbZ$, which since $\gcd(a,b)=1$, always has a solution.
\end{proof}
Now to prove \cref{thm:allclasses}, it suffices to construct examples which are homotopic to the linear map $B$ of the form in the preceding lemma.

We remark that the deformation approach to obtain $f$ from the map $f_0$ in \cref{sec:eg} changed the homotopy class of the map. For our general example, we will adapt this approach to construct an example in a desired homotopy class. The idea is to define the initial map $f_0$ with two invariant annuli, and then apply two shears in opposing directions along each annulus. To encourage visualising this idea, we refer the reader to \cref{fig:tan2}, which shows how the centre curves look for our example in the case $t=0$. We suggest comparing this to the centre curves of the example constructed in the previous section, as shown in \cref{fig:tan}. The depiction in \cref{fig:tan2} will be justified later. Now, we proceed to construct the examples.

\begin{lem}\label{lem:g-existence}
	Given integers $\mu, \lambda>1$, there exists $a\in(0,1/4)$ and a smooth, monotone circle map $g:\bbS^1\to \bbS^1$ such that $g$:
	\begin{itemize}
		\item is homotopic to the map $x\mapsto \mu x$,
		\item has fixed points $x=-2a,-a,0,\,a,\,2a$ and no other fixed points in $[-2a,2a]$,
		\item satisfies $g'(a)=g'(-a)<1$ and $g'(0)>\lambda$,
		\item is linear on the complement of $(-2a, 2a)$, on which we have $g' > \lambda$.
	\end{itemize}
	\end{lem}
	\begin{proof}
	Since we want $g$ to fix $[-2a,2a]$, then to ensure $g$ is homotopic to $x\mapsto \mu x$ and linear on $\bbS^1\setminus [-2a,2a]$, we require that the linear portion of $g$, considered as a $\bbZ$-periodic map on $\bbR$, maps $[2a,1-2a]$ to $[2a,\mu-2a]$. The slope of $g$ on this linear region is then $\frac{\mu-4a}{1-4a}$. Thus, by choosing $a$ sufficiently close to $1/4$, we can ensure $g'>\lambda$. After fixing such an $a$, we can then define $g$ over $[-2a,2a]$ to satisfy the second and third properties in a similar fashion to how $g$ was chosen in the earlier construction.
	\end{proof}
\begin{figure}[h]
\includegraphics{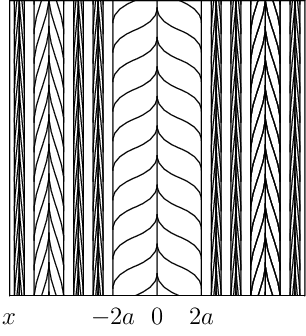}
	\caption{The centre curves of the general example when $t=0$ and $\mu=3$.}
	\label{fig:tan2}
\end{figure}
\begin{prop}
	If $A:\bbT^2\to\bbT^2$ is an expanding linear map with integer eigenvalues, there exists a partially hyperbolic surface endomorphism $f:\bbT^2\to\bbT^2$ which admits a periodic centre annulus and is homotopic to $A$.
\end{prop}
\begin{proof}
	Begin by letting $\lambda,\,\mu >1$ and $t\geq 0$ be integers, and let $B$ be as in \cref{lem:triangle}. We will later explain how an example can be constructed for non-positive $\lambda$ and $\mu$, while when $t<0$, the examples are similar, so the details are left to the reader.
	
	For our fixed $\lambda$ and $\mu$, let $a\in\bbS^1$ and $g:\bbS^1\to\bbS^1$ be as in \cref{lem:g-existence}, and define $f_0:\bbT^2\to\bbT^2$ by $f_0(x,y) = (g(x),\lambda y)$. Then $f_0$ is homotopic to the linear map given by $\begin{psmallmatrix}
			\mu & 0 \\
			0 & \lambda
			\end{psmallmatrix}$ and has two invariant annuli $(-2a,0)\times \bbS^1$ and $(0,2a)\times\bbS^1$ which share the boundary circle $\{0\}\times\bbS^1$. 
	
	Let $\varphi:\bbS^1\to\bbS^1$ be a smooth, monotone map such that $\varphi(0) = 0$, $\varphi(a)=(t+1)/2$ and $\varphi(2a)=t+1$, the support of $\varphi'$ is contained in $(0,2a)$ and $\varphi'(a)>1$. Define another smooth, monotone map $\psi:\bbS^1\to\bbS^1$ to be such that $\psi(-2a) = 0$, $\psi(-a) = -1/2$ and $\psi(0) = -1$, the support of $\psi'$ is contained in $(-2a,0)$ and $\psi'(-a)<-1$. Then $\varphi$ is a shear upwards by a factor of $t+1$ in one invariant annulus, while $\psi$ is a shear downwards by a factor of $1$ in the other. The explicit deformation to give the desired example is given by $f(x,y) = (g(x), \lambda y + \varphi(x) + \psi(x))$. Note that while the initial map $f_0$ was not homotopic to $B$ unless $t=0$, the shearing by both $\varphi$ and $\psi$ together result in $f$ being homotopic to $B$.
	
	To see that $f$ is partially hyperbolic, we adapt the main ideas of \cref{sec:eg}. Note that the orbits of points in the annuli $X_1 =(-2a,0)\times\bbS^1$ and $X_2 = (0,2a)\times\bbS^1$ are disjoint, so the approach is to define cone families much like the one annulus for the concrete example on each of them. The invariant circles $\{-a\}\times\bbS^1$ and $\{a\}\times \bbS^1$ are hyperbolic invariant manifolds, so there is a natural unstable cone-family defined on each of these circles akin to $\Cone_{\eps}$ in \cref{coneeps}. Meanwhile, on the complement of $X_1 \cup X_2$, the map $f$ is linear, and a cone family which is a small uniform neighbourhood of the horizontal will be an unstable cone family, similar to $\Bone$ in \cref{bone}. Arguing as in \cref{compatible} and \cref{glue}, we can stitch these cones together to construct $\Cone^u$, a cone family which satisfies $Df \Cone^u(p)\subset \sint\Cone^u(f(p)))$. Since $f$ expands $\Cone^u$ on the linear region close to the invariant hyperbolic circles, then by using arguments of \cref{ph}, we can show that $\Cone^u$ is in fact an unstable cone family and that $f$ is partially hyperbolic. Now by applying the arguments of \cref{invtannulus} to each invariant annulus $X_1$ and $X_2$, we see that they are invariant centre annuli of $f$.
	\medskip{}
	
    Now we drop the assumption that $\mu,\lambda >0$. So let the matrix $B$ be of the form $\begin{psmallmatrix}
			\mu & 0 \\
			t & \lambda
			\end{psmallmatrix}$
    where one or both of $\mu$ and $\lambda$ may be negative. Since $|\mu| > 1$, the linear map $g_0 : \bbS^1 \to \bbS^1,$ defined by $x \mapsto \mu x$
    has at least one point
    with period exactly two. Using a deformation,
    one can define a map $g : \bbS^1 \mapsto \bbS^1$
    and interval $I \subset \bbS^1$ such that
    such that
    \begin{itemize}
        \item 
        $g(I)$ is disjoint from $I$,
        \item
        $g^2(I)$ is equal to $I$, and
        \item
        $g$ is linear on the complement of $I$ with derivative
        $|g'| > |\lambda|$.
    \end{itemize}
    Up to conjugation with a rigid rotation,
    we may assume $I$ is centred at zero.
    That is, there is $a > 0$ such that $I = (2a, -2a)$.
    Moreover, by replacing $g|_I$
    (but leaving $g$ on $g(I)$ linear and unchanged),
    we may assume that $g^2$ has fixed points at
    $x = -2a, -a, 0, a, 2a,$
    and that $(g^2)'(-a) = (g^2)'(a) < 1$.
    In other words,
    $g^2$ here has the properties
    that $g$ had in the case when $\mu$ and $\lambda$ were assumed positive in the earlier section of this proof.  Let the shearing functions $\varphi$ and $\psi$
    be defined exactly as before and define
    $f(x,y) = (g(x), \lambda y + \varphi(x) + \psi(x)).$
    By again adapting the previous techniques,
    one can show that the resulting endomorphism
    is partially hyperbolic. It is easy to see that the annulus $I\times \bbS^1$ is an invariant centre annulus for $f^2$, so that $I\times \bbS^1$ is a periodic centre annulus for $f$.
\end{proof}

To establish \cref{thm:allclasses}, we now show that the examples constructed in the preceding proof are dynamically incoherent. While it is unclear whether or not the original example of \cref{sec:eg} was incoherent, the presence of two adjacent periodic annuli as depicted in \cref{fig:tan2} makes observing coherence straightforward.
	\begin{proof}[Proof of \cref{thm:allclasses}]
        Let $f$ be an example established in the proof of the preceding proposition with the assumption that $\lambda$, $\mu$, $t\geq0$. The other cases are again similar and left to the reader. When we restrict $f$ to $X_1$, the map admits an invariant splitting $E^c\oplus E^u$ and the circle $\{a\}\times\bbS^1$ is a hyperbolic attractor tangent to the unstable direction. The centre direction is thus uniquely integrable on a neighbourhood of this circle, which in turn implies it is uniquely integrable on all of $X_1$. Similarly, $E^c$ is uniquely integrable on $X_2$.

        Using the ideas of \cref{lem:negativeslope}, one can show that $E^c$ has positive slope on $X_1 = (-2a,0)\times\bbS^1$, while it has negative slope on $X_2 = (0,2a)\times\bbS^1$. Due to the shearing being in the opposite direction on each annulus, the centre curves approach the centre circle $\{0\}\times \bbS^1$ with slopes of opposite sign, as is shown in \cref{fig:tan2}. On a neighbourhood of the circle $\{0\}\times\bbS^1$ there cannot exist a foliation chart. Hence $f$ is dynamically incoherent.
	\end{proof}
	Note that the examples constructed in the preceding theorem do not admit invariant unstable directions. This can be seen by an argument similar to how we established this property for the earlier explicit example. Namely, if we had a unique unstable direction $E^u$, then for a point $p \in \{a\}\times\bbS^1$, one can show that $E^u(p)$ must be vertical. However, if $q$ is a point in the linear region of $f$ such that $f(q)=p$, the cone $Df (\Cone^u(q))\subset T_p\bbT^2$ will be a small neighbourhood of the horizontal which does not contain $E^u(p)$, which is a contradiction.

	We conclude by justifying the rest of the depiction of the centre curves as is shown in \cref{fig:tan2}. Once more, outside the orbits of the two invariant annuli $X_1$ and $X_2$, the map $f$ is a linear map preserving the horizontal and vertical directions, expanding stronger in the horizontal. An annulus in the preimage of either $X_1$ and $X_2$ is a linear rescaling of the curves on the invariant annulus that is contracted a greater amount in the horizontal than the vertical, and so when $t=0$ in the linearisation $B$, the centre curves are as shown. Note that the slopes on each annulus $X_1$ and $X_2$ will not be symmetric as in the figure when $t\neq 0$, though they will always be of opposite sign.
	
	\printbibliography
\end{document}